\documentclass{amsart}[12pt]

\def\volume{\operatorname{vol}}

\def\op{\operatorname}
\def\svolball#1#2{{\volume(\underline B_{#2}^{#1})}}

\usepackage[usenames,dvipsnames]{color}
\usepackage{amsmath,amssymb,bm,amsthm, amsfonts, mathrsfs, eucal, epsfig}
\usepackage{graphicx}
\usepackage{url}
\usepackage[top=2.5cm,bottom=2.0cm,left=2.5cm,right=2.5cm]{geometry}

\makeatletter 
\@addtoreset{equation}{section}
\makeatother  

\begin{document}

\newtheorem{Thm}{Theorem}[section]
\newtheorem{Def}[Thm]{Definition}
\newtheorem{Lem}[Thm]{Lemma}
\newtheorem{Rem}{Remark}[section]

\newtheorem{Cor}[Thm]{Corollary}
\newtheorem{sublemma}{Sub-Lemma}
\newtheorem{Prop}[Thm]{Proposition}
\newtheorem{Example}{Example}[section]
\newcommand{\g}[0]{\textmd{g}}
\newcommand{\pr}[0]{\partial_r}
\newcommand{\dif}{\mathrm{d}}
\newcommand{\bg}{\bar{\gamma}}
\newcommand{\md}{\rm{md}}
\newcommand{\cn}{\rm{cn}}
\newcommand{\sn}{\rm{sn}}
\newcommand{\seg}{\mathrm{seg}}

\newcommand{\Ric}{\mbox{Ric}}
\newcommand{\Iso}{\mbox{Iso}}
\newcommand{\ra}{\rightarrow}
\newcommand{\Hess}{\mathrm{Hess}}
\newcommand{\RCD}{\mathsf{RCD}}

\title{Quantitative maximal volume entropy rigidity on Alexandrov spaces}
\author{Lina Chen}
\address[Lina Chen]{Department of mathematics, Nanjing University, Nanjing China}

\email{chenlina\_mail@163.com}
\thanks{Supported by the NSFC 12001268 and a research fund from Department of Mathematics in Nanjing University.} 

\maketitle

\begin{abstract}

\setlength{\parindent}{10pt} \setlength{\parskip}{1.5ex plus 0.5ex
minus 0.2ex} 
  We will show that the quantitative maximal volume entropy  rigidity holds on Alexandrov spaces. More precisely, given $N, D$, there exists $\epsilon(N, D)>0$, such that for $\epsilon<\epsilon(N, D)$, if $X$ is an $N$-dimensional Alexandrov space with curvature $\geq -1$, $\op{diam}(X)\leq D, h(X)\geq N-1-\epsilon$, then $X$ is Gromov-Hausdorff close to a hyperbolic manifold. This result extends the quantitive maximal volume entropy rigidity of  \cite{CRX} to Alexandrov spaces. And we will also give a quantitative maximal volume entropy rigidity for $\op{RCD}^*$-spaces in the non-collapsing case.
  
  \end{abstract}
\section{Introduction}

Volume entropy is a geometric invariant which measures the asymptotic exponential growth rate of the volumes of metric balls in the universal cover of a compact manifold. Precisely, for $M$ a compact $N$-manifold, the volume entropy of $M$ is defined as
$$h(M)=\lim_{R\to \infty}\frac{\ln\volume(B_R(\tilde p))}{R},$$
where $\tilde p\in \tilde M$, the universal cover of $M$. By \cite{Ma}, the limit always exists and independent of $\tilde p$. 

For a compact $N$-manifold $M$ with $\op{Ric}_M\geq -(N-1)$, by Bishop volume comparison, $h(M)\leq N-1$. And when $h(M)=N-1$, Ledrappier-Wang \cite{LW} showed that $M$ is isometric to a hyperbolic manifold. We call this result the maximal volume entropy rigidity. When $h(M)$ is close to $N-1$, in \cite{CRX}, with Rong and Xu, we showed the quantitative maximal volume entropy rigidity: $M$ is diffeomorphic and is Gromov-Hausdorff close to a hyperbolic manifold. For non-smooth metric spaces, in \cite{Ji}, Jiang generalized the maximal volume entropy rigidity to Alexandrov spaces; Later,  in \cite{CDNPSW}, Conell-Dai-N\'u\~ne.Zimbr\'on-Perales-Su\'arez.Serrato-Wei showed that  for a $\RCD^*(-(N-1), N)$-space $(X, d, m)$ the maximal volume entropy rigidity holds and they also pointed out that the quantitative maximal volume entropy rigidity is right if in addition the systole has uniform lower bound, i.e., $\inf\{\tilde d(\tilde x, \gamma\tilde x), \tilde x\in \tilde X, \gamma\in \bar\pi_1(X)\setminus\{e\}\}\geq l>0$, where $(\tilde X, \tilde d, \tilde m)$ is the universal cover of $(X, d, m)$ (the existence is proved in \cite{MW}, see Theorem 2.9), and $\bar \pi_1(X)$ is the deck transform group.  At the same time they conjecture that the systole condition is not necessary. For a compact $\RCD^*(K, N)$ space $(X, d, m)$, the volume entropy is defined as 
$$h(X)=\limsup_{R\to \infty}\frac{\ln \tilde m(B_R(\tilde x))}{R},$$
where $\tilde x\in \tilde X$ and this limit is independent  of $\tilde x$ and $\tilde m$ (cf. \cite{Rev, BCGS}).

In this note, we will show the quantitative maximal volume entropy rigidity on Alexandrov spaces:

\begin{Thm}
Given $N>1, D>0$, there exists $\epsilon(N, D)>0$ such that for $0<\epsilon<\epsilon(N, D)$, if a compact $N$-dimensional Alexandrov space $X$ with curvature $\geq -1$ satisfies that
$$h(X)\geq N-1-\epsilon, \quad \op{diam}(X)\leq D,$$
then $X$ is $\Psi(\epsilon | N, D)$-Gromov-Hausdorff close to an $N$-dimensional hyperbolic manifold, where $\Psi(\epsilon | N, D)\to 0$ as $\epsilon\to 0$ and $N, D$ fixed. 
\end{Thm}
The idea of the proof of Theorem 1.1 is similar as Theorem D in \cite{CRX}. And from the proof below (Section 4), we could see that the idea can also be applied in $\RCD^*$-spaces.  In particular, for $\RCD^*$-spaces the quantitative maximal volume entropy rigidity holds in the non-collapsing case.

Recall that a $\op{RCD}^*(K, N)$-space $(X, d, m)$ is called non-collapsed if $m=\op{Haus}^N$, the $N$-dimensional Hausdorff measure and a sequence of non-collapsed $\RCD^*(K, N)$-spaces $(X_i, d_i, \op{Haus}^N, x_i)$ has a subsequence which is non-collapsed pointed measured Gromov-Hausdorff convergent  to $(X, d, \op{Haus}^N, x)\in \RCD^*(K, N)$ if and only if  $\op{Haus}^N(B_1(x_i))>v>0$ for some $v$ (see \cite[Definition 1.1, Theorem 1.2]{DePG}).
\begin{Thm} Given $N>1, D>0, v>0$, there exists $\epsilon(N, D, v)>0$ such that for $0<\epsilon<\epsilon(N, D, v)$, if a metric measure space $(X, d, \op{Haus}^N, x)\in \op{RCD}^*(-(N-1), N)$ satisfies 
$$h(X)\geq N-1-\epsilon, \quad \op{diam}(X)\leq D,\quad \op{Haus}^N(X)\geq v,$$
then $X$ is $\Psi(\epsilon | N, D, v)$-Gromov-Hausdorff close to an $N$-dimensional hyperbolic manifold.
\end{Thm}
To compare \cite[Theorem 1.4]{CDNPSW} and Theorem 1.2, note that a collapsing sequence may have systole uniform lower bound and a non-collapsing sequence may have systole going to 0 (see \cite[Remark 6.2]{PW}). Both may happen even for  Riemannian manifolds with uniform Ricci curvature lower bound. 

Since an $N$-dimensional Alexandrov space with curvature $\geq k$ is also an essentially non-branching $\RCD^*((N-1)k, N)$-space (\cite{Pe}, \cite{ZZ}, \cite{AGS}), in the following, we will discuss in the essentially non-branching $\RCD^*$-spaces except where we will use the generalized Margulis lemma proved by Xu-Yao \cite{XY} in Alexandrov spaces. 

Now, we briefly describe the things we need to do. First, we will point out that the equivariant Gromov-Hausdorff convergent structure results \cite[Theorem 3.10]{FY}  (see also \cite[Theorem 4.2]{FY1}) holds for metric measure space $\tilde X_i$ which is a universal cover of a compact $\RCD^*(K, N)$-space where they originally assume $\tilde X_i$ is simply connected (see section 3). Then we will give a almost warped product structure of the universal cover of a compact $\RCD^*(-(N-1), N)$-space which has almost maximal volume entropy. This result can be treated as a quantitative version of Theorem 1.2 in \cite{CDNPSW}: a functional type condition which implies an almost metric cone structure (see Section 4.1, Proposition 4.4).  Then for Theorem 1.1, we use Xu-Yao's generalized Margulis lemma in Alexandrov spaces to derive the discreteness of the limit group action. And last, notice that the free limit isometric action results (\cite[Theorem 2.1]{CRX}) holds in $\RCD^*(K, N)$-spaces, the proof is complete (see Section 4.2).

The author would like to thank Shicheng Xu for the recommendation of the topic of this note and Professor Guofang Wei's advisement about Theorem 1.2. 

\section{Prelminaries}
In this section, we will supply some notions and properties we need in the proof of Theorem 1.1 and Theorem 1.2. In the following, we always assume a metric measure space $(X, d, m)$ satisfies that the geodesic space $(X, d)$ is complete, separable and locally compact and $m$ is a nonnegative Radon measure with respect to $d$ and finite on bounded sets. We refer reader to the survey \cite{Am} for an overview of the topic and bibliography about curvature-dimension bounds in metric measure spaces.

\subsection{Calculus tools in metric measure spaces}

Consider a metric measure space $(X, d, m)$. A curve $\gamma: [0,1]\to X$ is called a constant speed geodesic if $d(\gamma(s), \gamma(t))=|s-t|d(\gamma(0), \gamma(1))$, for $s, t\in [0,1]$.
Let $\op{Geo}(X)$ be the class of constant speed geodesics in $X$.  Let $C([0, 1], X)$ be the space of continuous curves with weak convergence topology and let $\mathcal{P}(C([0,1], X))$ be the space of Borel probability measures of $C([0,1], X)$. For each $t\in [0, 1]$, define the evaluation map $e_t: C([0, 1], X) \to X$ by $e_t(\gamma)=\gamma(t)$. We say $\pi\in \mathcal{P}(C([0,1], X))$ is a test plan if there is a constant $c>0$, such that
$$(e_t)_*(\pi)\leq c m, \forall \, t\in [0, 1], \quad \int\int_0^1 |\dot\gamma(t)|dt d\pi(\gamma)<\infty,$$
where $|\dot\gamma(t)|=\lim_{h\to 0}d(\gamma(t+h), \gamma(t))/|h|$. Let  $S^2(X, d, m)$ be the set of $f: X\to \Bbb R$, such that there exists $G\in L^2(X, m)$,
 $$\int |f(\gamma(1))-f(\gamma(0))|d\pi(\gamma)\leq \int\int_0^1G(\gamma(t))|\dot\gamma(t)|dt d\pi(\gamma), \,   \forall \, \text{test plan } \pi, $$
where $G$ is called a weak upper gradient of $f$. Let $|\nabla f|_w$ be the minimal (in $m$-a.e. sense) weak upper gradient of $f$. The Cheeger energy defined on $S^2(X, d, m)$ can be written as 
$$Ch(f)=\frac{1}{2}\int |\nabla f|_w^2dm.$$
\begin{Def}
Let $W^{1,2}(X, d, m)=L^2(X, m)\cap S^2(X, d, m)$ endowed with the norm
$$\|f\|^2_{W^{1,2}}=\|f\|_{L^2}^2+ 2Ch(f).$$
We say $(X, d, m)$ is infinitesimally Hilbertian if $W^{1,2}(X, d, m)$ is an Hilbert space, i.e., the Cheeger energy is a quadratic form.
\end{Def}
In this subsection, we always assume that $(X, d, m)$ is an infinitesimally Hilbertian space. 

For an open subset $\Omega\subset X$, let $W^{1,2}_{\op{loc}}(\Omega)$ be the space of function $f: \Omega\to \Bbb R$ that locally equal to some function in $W^{1, 2}(X, d, m)$. For $f, g\in W^{1,2}_{\op{loc}}(\Omega)$, define
$$\Gamma(f, g)=\liminf_{\epsilon\to 0}\frac{|\nabla (g+\epsilon f)|_w^2-|\nabla g|_w^2}{2\epsilon}.$$
In fact $\Gamma(f, g)$ can be achieved by taking limit directly (cf. \cite{Gi2}). 
The map $\Gamma: W^{1,2}_{\op{loc}}(\Omega)\times W^{1,2}_{\op{loc}}(\Omega)\to L^1_{\op{loc}}(\Omega)$ is symmetric, bilinear and $\Gamma(f, f)=|\nabla f|_w^2$. Let $D(\Delta, \Omega)$ be the space of $f\in W^{1,2}_{\op{loc}}(\Omega)$ satisfies that there exists a Radon measure $\mu$ on $\Omega$ such that 
$$-\int\Gamma(f, g)=\int gd\mu$$
holds for any Lipschitz function $g: \Omega\to \Bbb R$, $\op{supp}g\subset\subset \Omega$. $\mu$ is unique and we call it the distributional Laplacian of $f$ and denote it by $\left.\Delta f\right|_{\Omega}$. By the definition, it is obvious that $D(\Delta, \Omega)$ is a vector space and the Laplacian is linear. If $f\in W^{1,2}(X, d, m)\cap  D(\Delta, X)$ and $\Delta f= hm$ for some $h\in L^2(X, m)$, we say $f\in D(\Delta)$ and denote $\Delta f=h$.
\begin{Prop}[\cite{Gi}]
Assume $(X, d, m)$ is infinitesimally Hilbertian. For an open set $\Omega\subset X$, a Lipschitz function $f: \Omega\to \Bbb R$, if there exists a Radon measure $\mu$ on $\Omega$ such that for any Lipschitz function $g: \Omega\to \Bbb R_{\geq 0}$ with $\op{supp}g\subset\subset \Omega$,
$$-\int\Gamma(f, g)dm\leq \int gd\mu,$$
then $f\in D(\Delta, \Omega)$ and $\Delta f\leq \mu$.
\end{Prop}

\subsection{$\RCD^*(K, N)$-spaces} Here we quickly recall some basic definitions and properties of $\RCD^*(K, N)$-spaces. 

The notion of metric measure spaces with curvature bounded below dimension bounded above which denoted by $\op{CD}$-condition were introduce by Lott-Villani (\cite{LV}) and Sturm (\cite{St1, St2}) independently. 

For a metric measure space $(X, d, m)$, let $\mathcal{P}_2(X)$ be the space of Borel probability measures $\mu$ on $(X, d)$ satisfying $\int_X d(x_0, x)^2d\mu(x)<\infty $ for some $x_0\in X$. For $\mu, \nu\in \mathcal{P}_2(X)$, define
$$W_2(\mu, \nu)=\left(\inf\int\int_0^1|\dot\gamma(t)|^2dt d\pi(\gamma)\right)^{\frac12},$$
where the infimum is taken among all $\pi\in \mathcal{P}(C([0,1], X))$ with $(e_0)_*(\pi)=\mu$, $(e_1)_*(\pi)=\nu$. The minimal is  always exists and is concentrated on $\op{Geo}(X)$. We call the plan $\pi$ which achieves the minimal an optimal transportation and denote the set of optimal transportations by $\op{OpGeo}(\mu, \nu)$. In fact, $W_2$ is a distance on $\mathcal{P}_2(X)$, and $(\mathcal{P}_2(X), W_2)$ is a geodesic space provided $(X, d)$ is a geodesic space (cf. \cite{Vi}).

Given a function $\phi: X\to \Bbb R\cup \{-\infty\}$ not identically $-\infty$, its $c$-transform $\phi^c: X\to \Bbb R\cup \{-\infty\}$ is defined as 
$$\phi^c(x)=\inf_{y\in X}\frac{d^2(x, y)}{2}-\phi(y).$$
We call $\phi$ is $c$-concave if $\phi^{cc}=\phi$.

For $N\geq 1, K$, let $\sigma_{K, N}: [0, 1]\times \Bbb R^{+}\to \Bbb R$ be as
$$\sigma_{K, N}^{t}(\theta)=\left\{\begin{array}{cc}
+\infty, & K\theta^2\geq N\pi^2,\\
\frac{\sin(t\theta\sqrt{K/N})}{\sin(\theta\sqrt{K/N})}, & 0<K\theta^2<N\pi^2,\\
t, & K\theta^2=0,\\
\frac{\sinh(t\theta\sqrt{-K/N})}{\sinh(\theta\sqrt{-K/N})},& K\theta^2<0.\end{array}\right.$$
and let 
$$\tau_{K, N}^{t}(\theta)=t^{\frac1N}\sigma_{K, N-1}^t(\theta)^{\frac{N-1}{N}}.$$

\begin{Def}
Given $K\in \Bbb R, N\geq 1$, we say a metric measure space $(X, d, m)$ is a $\op{CD}(K, N)$-space if for any two measures $\mu_0, \mu_1\in \mathcal{P}_2(X)$ with bounded support which contains in $m$'s support, there exists $\pi\in \op{OpGeo}(\mu_0, \mu_1)$ such that for each $t\in [0, 1]$
$$-\int\rho_t^{1-\frac{1}{N}}dm\leq -\int \tau_{K, N}^{1-t}(d(\gamma(0),\gamma(1)))\rho_0^{-\frac{1}{N}}(\gamma(0))+\tau_{K, N}^t(d(\gamma(0),\gamma(1)))\rho_1^{-\frac{1}{N}}(\gamma(1))d\pi(\gamma),$$
where $(e_t)_{*}\pi=\rho_tm + \mu_t, \mu_t\bot m$. We call $(X, d, m)$ is a $\op{CD}^*(K, N)$-space if the above inequality holds for $\sigma^t_{K, N}$ instead of $\tau^t_{K, N}$.
\end{Def}

We say $(X, d, m)$ is essentially non-branching if for $\mu, \nu\in \mathcal{P}_2(X)$ with bounded support, each $\pi\in\op{OpGeo}(\mu, \nu)$ is concentrated on a Borel set of non-branching geodesics. In the following, we always assume a metric measure space is essentially non-branching. 

Let 
$$\op{sn}_H(r)=\left\{\begin{array}{cc} \frac{\sin\sqrt H r}{\sqrt H}, & H>0;\\
r, & H=0;\\
\frac{\sinh\sqrt{-H}r}{\sqrt{-H}}, & H<0.\end{array}\right.$$
If $\op{CD}(K, N)$ holds locally on a family of sets covering $X$, we call $X$ satisfies $\op{CD}_{\op{loc}}(K, N)$. A question is if $\op{CD}_{\op{loc}}(K, N)$ implies $\op{CD}(K, N)$. It is known by \cite{CM}, $\op{CD}_{\op{loc}}^*(K, N)$ is equal to $\op{CD}^*(K, N)$.  And at the local level $\bigcap_{K'<K}\op{CD}_{\op{loc}}^*(K', N)$ coincide with $\bigcap_{K'<K}\op{CD}_{\op{loc}}(K', N)$. Then by \cite{CS}, \cite{Oh}, the following holds
  
\begin{Thm}[Generalized Bishop-Gromov relative volume comparison]
Let $(X, d, m)$ be a $\op{CD}^*(K, N)$-space, $N>1$. For $x\in X$, $r>0$, let $s_m(x, r)=\limsup_{\delta\to 0}m(\overline{B_{r+\delta}(x)}\setminus B_r(x))/\delta$. The following holds

$$\frac{s_m(x, r)}{\op{sn}_H^{N-1}(r)} \text{ and }\frac{m(B_r(x))}{\int_0^r\op{sn}_H^{N-1}(t)dt}$$
are non-increasing in $r$, where $r<\frac{\pi}{\sqrt H}$, if $H=\frac{K}{N-1}>0$. 
\end{Thm}

\begin{Def} A metric measure space $(X, d, m)$ is a $\RCD^*(K, N)$-space  if it is an infinitesimally Hilbertian $\op{CD}^*(K, N)$-space.
\end{Def}
In $\RCD^*(K, N)$-space, by \cite{Gi1}, \cite{CMo}, we have that
\begin{Thm}[Generalized Laplacian comparison]
For $\op{RCD}^*(K, N)$-space $(X, d, m)$, let $o\in X$ and let $r(x)=d(x, o)$. Then $r\in D(\Delta, X\setminus\{o\})$ and 
$$\Delta r\leq  \sqrt{|K|(N-1)}\cdot\frac{\op{sn}'_H(r)}{\op{sn}_H(r)}m.$$
\end{Thm}

In $\RCD^*(K, N)$-space, the splitting theorem and volume cone rigidity have been proved:
\begin{Thm}[\cite{Gi}]
Let $(X, d, m)$ be a $\RCD^*(0, N)$-space, $N>1$. If $X$ contains a line, then $(X, d, m)$ splits, i.e., $(X, d, m)$ is isometric to $(\Bbb R\times Y, d_{\Bbb R}\times d', \op{Haus}^1\times m')$, where $\op{Haus^1(\cdot)}$ is $1$-dimensional Hausdorff measure and $(Y, d', m')\in \RCD^*(0, N-1)$.
\end{Thm}
\begin{Thm}[\cite{PG}]
Let $(X, d, m)$ be a $\RCD^*(0, N)$-space, $x\in X$, $N>1$. If there are $R>r>0$ such that 
$$\frac{m(B_r(x))}{m(B_R(x))}=\left(\frac{r}{R}\right)^N,$$
then $B_R(x)$ is isometric to a ball $B_R(x^*)\subset (C(Z), x^*)$, where $C(Z)$ is a metric cone over $Z$, $Z\in \RCD^*(N-2, N-1)$, $x^*$ is the cone point.
\end{Thm}
Note that for $\RCD^*(K, N)$-spaces, $K\neq 0$, there are corresponding volume cone rigidity results as in above theorem.

For a metric space $(X, d)$, we say a connected covering space $\tilde{\pi}: (\tilde X, \tilde d)\to (X, d)$ is a universal cover of $(X, d)$ if for any other covering $\pi: (Y, d')\to (X, d)$, there is a continuous map $f: \tilde X\to Y$, such that $\pi\circ f=\tilde{\pi}$. A universal cover of a metric space may not exist in general (see \cite[Example 17]{Sp}). In \cite{MW}, Mondino-Wei showed that each $\RCD^*(K, N)$-space has a universal cover.
\begin{Thm}[\cite{MW}]
If a metric measure space $(X, d, m)\in \RCD^*(K, N)$, $K\in \Bbb R, N\geq 1$, then $(X, d, m)$ has a universal cover space $(\tilde X, \tilde d, \tilde m)\in \RCD^*(K, N)$.
\end{Thm}
In the above theorem, we take $\tilde d, \tilde m$ such that $\tilde \pi: \tilde X\to X$ is distance and measure non-increasing and is a local isometry (the existence of $\tilde d$ see \cite{Ri}). 

For a metric space $(X, d)$ which has a universal cover $(\tilde X, \tilde d)$, let $\bar \pi_1(X)$ be the revised fundamental group of $X$, i.e., the deck transformations of $\tilde X$ and $\bar\pi_1(X)$ acts on $(\tilde X, \tilde d)$ isometrically. For $\bar\pi_1(X)$ we have that: (i) for each $\alpha\in \bar \pi_1(X)$, there is $\gamma\in \pi_1(X, x)$ which can be treated as a deck transformation on $\tilde X$ and denote it by $\Phi(\gamma)$, such that $\alpha=\Phi(\gamma)$, i.e. there is a surjective $\Phi: \pi_1(X, x)\to \bar \pi_1(X)$; (ii) If $\tilde X$ is simply connected, then $\Phi$ is an isometry, i.e., $\bar\pi_1(X)=\pi_1(X, x)$. By Gromov, the fundamental group of a compact Riemannian manifold is finitely generated, for the revised fundamental group of a compact $\RCD^*(K, N)$-space, the same result holds. 
\begin{Lem}
Let $(X, d, m)\in \RCD^*(K, N)$, $K\in \Bbb R, N\geq 1$, with $\op{diam}(X)\leq D$. Then $\bar\pi_1(X)$ is finitely generated by loops of length $< 3D$. 
\end{Lem}
\begin{proof}
By the proof of \cite[Theorem 2.7, Theorem 3.1]{MW}, we know that the universal cover $\tilde X$ is isometric to a $\delta$-cover $\tilde X^{\delta}$ of $X$ for some $\delta>0$, whose fundamental group $\pi_1(\tilde X^{\delta}, \tilde p)$ is isometric to the group generated by $\alpha^{-1}\beta\alpha$, where $\alpha$ is a path from $p$ to $\beta(0)$, $\beta$ is a loop in $B_{\delta}(q)$, for some $q\in X$.

Take a maximal subset $A$ of $X$ such that $p\in A$, $x_1=p$ and for each $x_i, x_j\in A, x_i\neq x_j$, $d(x_i, x_j)\geq \frac{\delta}{8}$. By Generalized Bishop-Gromov relative volume comparison Theorem 2.4, $|A|\leq c(N, K, \delta, D)$. Let $e_{i, j}$ be the shortest geodesic between $x_i, x_j$ when $d(x_i, x_j)\leq \frac{\delta}{2}$. For each 
$\alpha\in \bar\pi_1(X)$, let $\gamma\in \pi_1(X, p)$ be the projection of the shortest curve from $\tilde p$ to $\alpha(\tilde p)$ in $\tilde X$. Divide $\gamma$ into pieces of length $\leq \frac{\delta}{8}$ with cut points $p=\gamma(0)=\gamma(t_0), \gamma(t_1), \cdots, \gamma(t_m)$. For each $\gamma(t_i)$ there is $x_{k_i}$ such that $d(\gamma(t_i), x_{k_i})\leq \frac{\delta}{8}$, $x_{k_0}=x_1=p$. Take a shortest geodesic $c_i$ from $\gamma(t_i)$ to $x_{k_i}$. Then the loop $\left.\gamma\right|_{[t_i, t_{i+1}]}\ast c_{i+1}\ast e_{k_{i+1}, k_i}\ast c_i^{-1}\subset B_{\delta}(\gamma(t_i))$. Let $e=e_{1, k_1}\ast e_{k_2, k_3}\ast \cdots \ast e_{k_m, 1}$. Then
 $\gamma\ast e^{-1}\in \pi_1(\tilde X, \tilde p)$, i.e., $\Phi(e)=\alpha$. Let $h_j$ be a shortest geodesic from $p$ to $x_i$ and let $\gamma_{i, j}=h_i\ast e_{i, j}\ast h_j^{-1}$. Then $\bar\pi_1(X)$ can be generated by $\{\gamma_{i, j}\}$ with $|\gamma_{i,j}|< 2D+\delta$. Since $|A|$ is finite, $\bar\pi_1(X)$ is finitely generated.  
\end{proof}

\subsection{Gromov-Hausdorff convergence} In this subsection, we will state some definitions and properties about Gromov-Hausdorff convergence.

Consider a sequence of pointed geodesic spaces $(X_i, d_i, p_i)$ and $(X, d, p)$. We say $(X_i, d_i, p_i)$ is Gromov-Hausdorff convergent to $(X, d, p)$, denoted by $(X_i, d_i, p_i)\overset{GH}\to (X, d, p)$, if there are $\epsilon_i\to 0$ and maps $f_i: B_{\frac{1}{\epsilon_i}}(p_i)\to B_{\frac{1}{\epsilon_i}+\epsilon_i}(p)$, $f_i(p_i)=p$ that each is an $\epsilon_i$-Gromove-Hausdroff approximation (briefly, $\epsilon_i$-GHA),  i.e., satisfies the following two properties:

 (i) $\epsilon_i$-embedding: $|d_i(x_i, x'_i)-d(f_i(x_i), f_i(x'_i))|\leq \epsilon_i, \forall\, x_i, x'_i\in B_{\frac{1}{\epsilon_i}}(p_i)$;

(ii) $\epsilon_i$-onto: $B_{\frac{1}{\epsilon_i}}(p)\subset B_{\epsilon_i}(f_i(B_{\frac{1}{\epsilon_i}}(p_i)))$. 

Assume $X_i$ (resp. $X$) admits a closed isometric action $G_i$ (resp. $G$), we say that $(X_i, p_i, G_i)$ is equivariant Gromov-Hausdorff convergent to $(X, p, G)$ if there are $\epsilon_i\to 0$, $\epsilon_i$-GHAs $f_i$ and $\phi_i: G_i(\epsilon_i^{-1})\to G(\epsilon_i^{-1}+\epsilon_i)$, 
$\psi_i: G(\epsilon_i^{-1})\to G_i(\epsilon_i^{-1}+\epsilon_i)$, such that for $g_i\in G_i(\epsilon_i^{-1})$, $g\in G(\epsilon_i^{-1})$, $x_i\in B_{\frac{1}{\epsilon_i}}(p_i)$,
$$d(f_i(x_i), \phi_i(g_i)(f_i(g_i^{-1}(x_i)))\leq \epsilon_i,$$
$$d(g(f_i(x_i)), f_i(\psi_i(g)(x_i)))\leq \epsilon_i,$$
where $G_i(R)=\{g_i\in G_i, d_i(p_i, g_i(p_i))\leq R\}$.
\begin{Thm}[\cite{Fu}, \cite{FY}]
Assume $(X_i, p_i)\overset{GH}\to(X, p)$ and $X_i$ admits a closed isometric action $G_i$. Then there is a closed isometric group $G$ on $X$ such that by passing to a subsequence, $(X_i, p_i, G_i)$ is equivariant Gromov-Hausdorff convergent to $(X, p, G)$ and $(X_i/G_i, \bar p_i)\overset{GH}\to (X/G, \bar p)$, where $\bar p_i=\pi_i(p)$, $\pi_i: X_i\to X_i/G_i$ is the quotient map.
\end{Thm}

We say metric measure spaces $(X_i, d_i, m_i, p_i)$ is measured Gromov-Hausdorff (briefly mGH) convergent to $(X, d, m, p)$, if for each $R>0$, there are measurable $\epsilon_i$-GHAs $f_i : (B_R( p_i), p_i)\to (B_{R+\epsilon_i}( p), p)$ and 
$$(f_i )_{\ast}\left(\left.m_i\right|_{B_R(p_i)}\right)\to \left.m\right|_{B_{R+\epsilon_i}(p_i)}$$ in the weak topology of measures.

\section{Equivariant Gromov-Hausdorff convergent structure in $\op{RCD}^*$-spaces}

In the section, we will generalize \cite[Theorem 3.10]{FY} (see also \cite[Theorem 4.2]{FY1}) to spaces which may not be simply connected but universal covers of compact $\RCD^*(K, N)$-spaces. This structure of equivariant Gromov-Hausdorff convergent sequences will be used in the proof of non-collapsing of $X$ in Theorem 1.1 (see section 4.2).
 
In the following of this section, we always assume that a sequence of metric measure spaces $(X_i, d_i, m_i)$ is measured Gromov-Hausdorff convergent to $(X, d, m)$ and $X_i\in \RCD^*(K, N)$ for all $i$.  By the compactness of $\RCD^*(K, N)$-spaces, we have that $X$ is a $\RCD^*(K, N)$ space (\cite{AGS}). Assume $(\tilde X_i, \tilde d_i, \tilde m_i)$ is the universal cover of $(X_i, d_i, m_i)$ (Theorem 2.9) and let $\Gamma_i=\bar\pi_1(X_i)$ be the revised fundamental group which acts isometrically on $\tilde X_i$. By Theorem 2.11, passing to a subsequence, we have the following communicate diagram: 
\begin{equation}\begin{array}[c]{ccc}
(\tilde X_i,\tilde p_i,\Gamma_i)&\xrightarrow{GH}&(\tilde X,\tilde p,G)\\
\downarrow\scriptstyle{\pi_i}&&\downarrow\scriptstyle{\pi}\\
(X_i,p_i)&\xrightarrow{GH} &(X, p)
\end{array} \end{equation}
By \cite{GR} or \cite{So}, $\Gamma_i, G$ are Lie groups. Let $G_0$ be the identity component of $G$. Comparing with \cite[Theorem 3.10]{FY} (\cite[Theorem 4.2]{FY1}), we have that
\begin{Thm}
Let $X_i, \tilde X_i, \Gamma_i, \tilde X, X, G, G_0$ be as in (3.1). Assume $X$ is compact.
Then there are normal subgroups $\Gamma_{i\epsilon}\subset \Gamma_i$ such that 

{\rm (3.1.1)} $\Gamma_{i\epsilon}$ is generated by $\Gamma_i(\epsilon)$ for some $\epsilon>0$;

{\rm (3.1.2)} $(\tilde X_i,  \tilde p_i, \Gamma_{i\epsilon})$ is equivariant Gromov-Hausdorff convergent to  $(\tilde X, \tilde p, G_0)$;

{\rm (3.1.3)} for $i$ large, $\Gamma_i/\Gamma_{i\epsilon}$ is isometric to $G/G_0$.
 \end{Thm}

In \cite[Theorem 3.10]{FY}, they considered an equivariant Gromov-Hausdorff convergent sequence of metric spaces $(\tilde X_i, \tilde p_i, \Gamma_i)\to (\tilde X, \tilde p, G), G_0\subset G$ and derived the same results as in Theorem 3.1 by assumming that (1) $G/G_0$ is finitely presented; (2) $G/G_0$ is discrete; (3) $\Gamma_i$ is properly discontinuous and free; (4) $\tilde X_i$ is simply connected; (5) $G_0$ is generated by $G_0(R_0)$, for some $R_0>0$; (6) $\pi_1(B_{R_0}(\tilde p), \tilde p)$ is surjective in $\pi_1(\tilde X, \tilde p)$; (7) $\tilde X/G$ is compact.  And in \cite[Theorem 4.2]{FY1}, they showed that without the conditions (1) and (6) the same results hold. Here, in Theorem 3.1, the conditions (2), (3), (5), (7) holds obviously. The weaker condition that $\tilde X_i$ is a universal cover of a compact $\RCD^*(K, N)$-space plays the same role as that $\tilde X_i$ is simply connected in \cite[Theorem 4.2]{FY1}. In our application below, $\tilde X$ is the $k$-dimensional simply connected hyperbolic space form.

 The proof of Theorem 3.1 is along the same line as \cite[Theorem 4.2]{FY1}. As we have discussed above, compared with \cite[Theorem 4.2]{FY1}, in Theorem 3.1 the only weaker condition is that $\tilde X_i$ is a universal cover of a $\RCD^*$-space $X_i$ and $\tilde X_i$ may not be simply connected. And correspondingly the results of Theorem 3.1 are about the deck transformation group of $\tilde X_i$, $\Gamma_i$. In the following, we will give a rough proof of Theorem 3.1 after \cite[Theorem 3.10]{FY} and \cite[Theorem 4.2]{FY1}.
 
 \begin{proof}[Proof of Theorem 3.1]
 
 Step 1: Group construction. 
 
 This step is the same as in \cite{FY} where we use the facts that $(\tilde X_i, \tilde p_i, \Gamma_i)$ is equivariant GH-convergent to $(\tilde X, \tilde p, G)$, $G/G_0$ is discrete and $\op{diam}(X)=D<\infty$.  
 
 By definition we have $\epsilon_i$-approximations $(f_i, \phi_i, \psi_i)$: 
$$f_i: B_{\epsilon_i^{-1}}(\tilde p_i)\to B_{\epsilon_i^{-1}+\epsilon_i}(\tilde p), \quad f_i(\tilde p_i)=\tilde p,$$ 
$$\phi_i: \Gamma_i(\epsilon_i^{-1})\to G(\epsilon_i^{-1}+\epsilon_i),$$ 
$$\psi_i: G(\epsilon_i^{-1})\to \Gamma_i(\epsilon_i^{-1}+\epsilon_i).$$
By the definition of pointed equivariant GH convergence, we only have the relations between $\Gamma_i(R)$ and $G(R)$ for large $R$. For large $i$, fixed  $10D<R<\epsilon_i^{-1}$. Consider 
$$\Gamma'_i(R)=\{\gamma\in \Gamma_i(R), \phi_i(\gamma)\in G_0\}.$$
Let 
$$\Lambda_i(R)=\Gamma_i(R)/\Gamma'_i(3R), \, H(R)=G(R)/G_0(3R),$$
$$V_i(R)=B_R(\tilde p_i)/\Gamma'_i(3R), \, W(R)=B_R(\tilde p)/G_0(3R),$$
where $\Gamma_i(R)/\Gamma'_i(3R) =\Gamma_i(R)/\sim$, $\gamma\sim \gamma'$ iff $\gamma^{-1}\gamma'\in \Gamma'_i(3R)$ and $B_R(\tilde p_i)/\Gamma'_i(3R)= B_R(\tilde p_i)/\sim$, $x\sim y$ iff there is $\gamma\in \Gamma'_i(3R)$ such that $\gamma x=y$.

By the construction $\Lambda_i(R), H(R)$ are quotient pseudo-groups. Now construct group $\hat H$ as $\hat H= F_{H(R)}/N_{H(R)}$, where $F_{H(R)}$ is the free group generated by $\{e_{\gamma}, \gamma\in H(R)\}$ and $N_{H(R)}$ is the normal group generated by $\{e_{\gamma_1\gamma_2}e_{\gamma_2}^{-1}e_{\gamma_1}^{-1}$, $\gamma_1, \gamma_2, \gamma_1\gamma_2\in H(R)\}$. And define $\hat \Lambda_i=F_{\Lambda_i(R)}/N_{\Lambda_i(R)}$ similarly.

For non-complete pseudo-covering spaces $V_i(R)$ of $X_i$ and $W(R)$ of $X$, we construct space $W$ as the connected component contains $(e, \bar p)$ of $\hat H\times W(R)/\sim$, where $(g\gamma, x)\sim (g, \gamma x)$ for $g\in \hat H$, $\gamma\in H(3R), x, \gamma x\in W(R)$ and construct $V_i$ as a covering of $X_i$ similarly.

The quotient group actions $\hat H$ on $\hat H\times W(R)/\sim$ and $\hat \Lambda_i$ on $\hat\Lambda_i\times V_i(R)/\sim$ are defined as $g((h, x))=(gh, x)$. Let $H, \Lambda_i$ be the subgroups of $\hat H$ and $\hat \Lambda_i$ that preserving $W, V_i$ respectively.

Step 2: Find $\Gamma'_{iR}$ such that $\Gamma'_{iR}\cap \Gamma_i(R)=\Gamma'_i(R)$ and $\Gamma_i/\Gamma'_{iR}=\Lambda_i$.

This step is similar as in \cite{FY}.

By the constructions and the facts that $\Gamma_i$ acts on $\tilde X_i$ properly discontinuous and freely, for $(V_i, \Lambda_i)$ and $(W, H)$ we could show that: (i) The actions $\Lambda_i$ on $V_i$  and $H$ on $W$ are properly discontinuous; (ii) $\Lambda_i$ is a free action; (iii) $\Lambda_i$ is isometric to $H$; (iv) $V_i/\Lambda_i=X_i,  W/H=X$ (see \cite[Lemma A1.12 - A1.14]{FY}).

By (i) and (ii), we know that $V_i$ is a cover space of $X_i$ and thus $V_i$ is a $\RCD^*(K, N)$-space. By Theorem 2.9, $V_i$ has a universal cover space and it is $\tilde X_i$. 

Let $\bar \pi_i: V_i\to X_i$ be the covering map and let 
$$\Gamma'_{iR}=\bar\pi_{i *}(\bar\pi_1(V_i)).$$
Then by using $V_i=\tilde X_i/ \bar \pi_1(V_i), X_i=\tilde X_i/\Gamma_i$ and (iv), we have that $V_i$ is a normal cover and \begin{equation}\Gamma_i/\Gamma'_{iR}=\Lambda_i. \label{quo-group}\end{equation}
Note that in \cite{FY},  \eqref{quo-group} is derived by using the short exact sequence of the fundamental group and the simply connectedness of $\tilde X_i$.

To see $\Gamma'_{iR}\cap \Gamma_i(R)=\Gamma'_i(R)$, as in \cite[Lemma A1.15]{FY}, first note that for each $\gamma\in \Gamma'_i(R)$, $d(\gamma(\tilde p_i), \tilde p_i)<R$. By the discussion in front of Lemma 2.10, there is a corresponding loop $l_{\gamma}\subset V_i(R)\subset V_i$. Thus $\Gamma'_{iR}\supset  \Gamma'_i(R)$. For $\gamma\in \Gamma'_{iR}\cap \Gamma_i(R)$, there is loop $l_{\gamma}$ in $V_i$ such that 
$\bar\pi_{i*}(l_{\gamma})=\gamma$, and the length of $l_{\gamma}<R$.
Thus $l_{\gamma}\subset V_i(R)=B_R(\tilde p_i)/\Gamma'_i(3R)$ and $\gamma\in \Gamma'_i(R)$ which implies $\Gamma'_{iR}\cap \Gamma_i(R)\subset \Gamma'_i(R)$.

Step 3: Let $\Gamma''_{iR}$ be the group generated by $\Gamma'_i(R)$. Then $\Gamma''_{iR}=\Gamma'_{iR}$ and $\Gamma''_{iR}$ is generated by $\Gamma_i(\epsilon)$ for some $\epsilon>0$.

This step is similar as in \cite{FY1}. 

First, by Lemma 2.10, we have that for large $i$, $\Gamma_i$ is generated by $\Gamma_i(3D)$. And then $G$ is generated by $G(6D)$.

In fact, for each $\bar P(\gamma)\in \left(G/G_0\right)(R)$, where $\bar P: G\to G/G_0$ is the quotient map, for large $i$, there is $\gamma_{i1}, \cdots, \gamma_{iN}\in \Gamma_i(3D)$ such that
$$\psi_i(\gamma)=\gamma_{i1}\cdots\gamma_{iN},$$
where $N$ depends on $R$ and $D$.
Since for each $j=1,\cdots, N$,
$$d(\phi_{i}(\gamma_{ij})f_i(\tilde p_i), f_i(\gamma_{ij}(\tilde p_i)))\leq \epsilon_i,  \quad d(\gamma_{ij}(\tilde p_i), \tilde p_i)\leq 3D,$$
we have that
$$\phi_i(\gamma_{ij})\in G(3D+2\epsilon_i).$$
And 
$$d(\gamma(\tilde p), \phi_i(\gamma_{i1})\cdots\phi_{i}(\gamma_{iN})(\tilde p))<CN\epsilon_i.$$

Since $G/G_0$ is discrete, and without loss of generality, we may assume $\bar P(\gamma^{-1}\phi_i(\gamma_{i1})\cdots\phi_{i}(\gamma_{iN}))$ is not in the isotropy group at $\bar p\in \tilde X/G_0$, then for $i$ large,  $\bar \pi(\gamma)=\bar\pi(\phi_i(\gamma_{i1})\cdots\phi_{i}(\gamma_{iN}))$. 

Thus $\gamma=\phi_i(\gamma_{i1})\cdots\phi_{i}(\gamma_{iN})\gamma'$ for some $\gamma'\in G_0$, i.e., $G$ can be generated by $G(6D)$.

Second, we have that for large $R$ and $i$,
\begin{equation}\lim_{i\to\infty}(\tilde X_i, \tilde p_i, \Gamma''_{iR})\to (\tilde X, \tilde p, G_0) \label{con-lim}\end{equation}
and $\Gamma''_{iR}$ is generated by $\Gamma_i(\epsilon)$ for some $\epsilon>0$ which implies $\Gamma''_{iR'}=\Gamma''_{iR}$ for each $R'\geq R$.

The above facts are same as \cite[Lemma 5.8, Lemma 5.9]{FY1}. In fact, if 
$$\lim_{i\to\infty}(\tilde X_i, \tilde p_i, \Gamma''_{iR})\to (\tilde X, \tilde p, G''),$$
then $G''\supset G_0(R)$ and thus $G''\supset G_0$. Assume $\gamma\in G''\setminus G_0$, $\gamma\in G(R'), R'>R$ and 
$$\lim_{i\to\infty}(\tilde X_i, \tilde p_i, \Gamma''_{iR'})\to (\tilde X, \tilde p, G''').$$
Since $\Gamma'_{iR'}\supset \Gamma''_{iR'}$ and 
$$\Gamma'_{iR'}\cap \Gamma_i(R')=\Gamma'_i(R'),$$
we have
$\left<\Gamma'_i(R')\right>(R')\subset \Gamma_i(R')$. Note that  $\Gamma''_{iR'}(R')\to G'''(R')$, $\Gamma'_i(R')\to G_0(R')$. Thus $G'''(R')=G_0(R')$ and thus $\gamma\notin G'''$ which is a contradiction to $\Gamma'_i(R')\supset \Gamma'_i(R)$.

By \eqref{con-lim} and the fact that $G_0$ is generated by $G_0(\frac{\epsilon}{2})$, a similar argument as in the First part of Step 3 gives that $\Gamma''_{iR}$ is generated by $\Gamma_i(\epsilon)$ for some $\epsilon>0$ (see the proof of \cite[Lemma 5.9]{FY1} for the details).

Last, for $R$ and $i$ large, $\Gamma''_{iR}$ is normal and $\Gamma''_{iR}=\Gamma'_{iR}$.

To see this fact, just note that: (1) by the construction and the fact that $\Gamma_i$ is generated by $\Gamma_i(3D)$, $\Gamma'_{iR}\supset \Gamma''_{iR}$ and thus $\Gamma_i/\Gamma'_{iR}\subset \Gamma_i/\Gamma''_{iR}$; (2)
$\Lambda_i(R)$ generates $\Gamma_i/\Gamma''_{iR}$ and $\left<\Lambda_i(R)\right>\subset \Lambda_i\cong \Gamma_i/\Gamma'_{iR}$ (see the proof of \cite[Lemma 5.11, Lemma 5.12]{FY1} for more details).

Step 4: $H$ is isomorphic to $G/G_0$.

This step is the same as in \cite{FY1}.

By \cite[Lemma A1.10]{FY}, there is a homomorphism $P_R: H\to G/G_0$. And since $G$ is generated by $G(6D)$, $P_R$ is surjective. To see $P_R$ is injective, we use the fact that for $i$ large,
$$\Lambda_i\cong \Gamma_i/\Gamma'_{iR}\cong H,$$
and 
$$\Gamma_i\to G,\quad \Gamma'_{iR}=\Gamma''_{iR}=\Gamma''_{iR'}\to G_0, \text{ for } R'\geq R.$$
Let $\gamma$ be in the kernel of $P_R$. Then take $\gamma$ as an element of $\Gamma_i/\Gamma'_{iR}$ especially an element of $\Gamma_i(R')/\Gamma'(3R')$.
By \cite[Lemma A.1.9]{FY}, there is a bijection between $\Gamma_i(R')/\Gamma'_i(3R')$ and  $G(R')/G_0(3R')$, we can also treat $\gamma$ as an element of $G(R')/G_0(3R')\subset H$. Then the injection of $G(R')/G_0(3R')\to G/G_0$ implies $\gamma$ is identity (see the proof of \cite[Lemma 5.13]{FY1} for more details).
\end{proof}

\section{Proof of Theorem 1.1 and 1.2}

To prove Theorem 1.1 and Theorem 1.2, as the proof of Theorem D in \cite{CRX}, consider a sequence of $\RCD^*(-(N-1), N)$-spaces, $(X_i, d_i, m_i)$ which is measured Gromov-Hausdorff convergent to $(X, d, m)$ and satisfies that 
\begin{equation}\op{diam}(X_i)\leq D, \quad h(X_i)\geq N-1-\epsilon_i\to N-1\end{equation}
and the following commutative diagram
\begin{equation}\begin{array}[c]{ccc}
(\tilde X_i,\tilde p_i,\Gamma_i)&\xrightarrow{GH}&(\tilde X,\tilde p,G)\\
\downarrow\scriptstyle{\pi_i}&&\downarrow\scriptstyle{\pi}\\
(X_i,p_i)&\xrightarrow{GH} &(X, p)
\end{array} \end{equation}
where $\Gamma_i$ is the revised fundamental group of $(X_i, d_i)$. Then we only need to show that $(X, d, m)$ is isometric to a hyperbolic manifold if one of the following conditions holds:

(a) $\{X_i\}$ are Alexandrov spaces with curvature $\geq -1$;

(b) $\op{Haus}^N(X_i)>v>0$ for all $i$.

We will prove Theorem 1.1 and Theorem 1.2 in the following two steps: 1, $\tilde X$ is isometric to $\Bbb H^k$, for some $k\leq N$ (subsection 4.1); 2, under the condition (a) or (b), the isometric action $G$ is discrete and free (subsection 4.2).

We denote $\Psi(\epsilon_1, \cdots, \epsilon_i | \delta_1, \cdots, \delta_j)$ as a function that goes to $0$ as $\epsilon_1, \cdots, \epsilon_i$ go to $0$ and $\delta_1, \cdots, \delta_j$ are fixed. And let $-\kern-1em\int_{B_r(x)} f=\frac{1}{m(B_r(x))}\int_{B_r(x)}f$. 

\subsection{Almost warped product structure} 
The main result of this subsection is 
\begin{Thm} Let $\tilde X$ be as in (4.1) and (4.2). We have that $\tilde X$ is isometric to $\Bbb H^k$, $k\leq N$.
\end{Thm}

To prove Theorem 4.1, we will show that $\tilde X$  is isometric to a warped product space $\Bbb R\times_{e^r} Y$. Then that $\tilde X=\Bbb H^k$ will be derived by the same argument of Section 8 in \cite{CDNPSW} (see also \cite[Lemma 4.4]{CRX}). 

First recall that
\begin{Thm}\cite[Theorem 1.2]{CDNPSW}
Given $N>1$, a complete $\RCD^*(-(N-1), N)$-space $(X, d, m)$ is isometric to a warped product space $\Bbb R\times_{e^r} Y$, where $Y$ is a $\RCD^*(0, N)$-space, if there is a function $u\in D(\Delta, X\setminus \{pt.\})$ satisfying that $|\nabla u|_w=1$ m-a.e. and $\Delta u=N-1$.
\end{Thm}

In the following, we will show a quantitative version of Theorem 4.2 (see Proposition 4.4). And Theorem 4.1 will be derived by proposition 4.3 and 4.4.



\begin{Prop}
Let $(X, d, m)$ be a $\RCD^*(-(N-1), N)$-space, $\tilde p\in \tilde X$.  If 
$$\op{diam}(X)\leq D, \quad h(X)\geq N-1-\epsilon,$$
then for $R>2D$, there are $r_n\to \infty$,  $x_n\in B_D(\tilde p)$, such that 

{\rm (4.3.1)} there is $\tilde p_n$ ($=\gamma_n(\tilde p)$, for some $\gamma_n\in \Gamma=\bar\pi_1(X)$ here) such that $r_n-2R<d(\tilde p_n, x_n)<r_n+2R$ and $f_{n, R}(y)=d(y, \tilde p_n)-d(x_n, \tilde p_n)$ satisfies
$$-\kern-1em\int_{B_R(x_n)}\Delta f_{n, R}\geq N-1-\Psi(\epsilon, r_n^{-1} | N, D);$$

{\rm (4.3.2)} for each $y\in B_R(x_n)$ there exists $q_n(y)\in \tilde X$ such that $d(q_n(y), \tilde p_n)=r_n+5R$ and 
$$d(y, q_n(y))+d(y, \tilde p_n)\leq r_n+5R+\Psi(\epsilon, r_n^{-1} | N, D).$$
\end{Prop}

\begin{Prop}
Let $(X, d, m)$ be a complete $\RCD^*(-(N-1), N)$-space. Given $D>0$, for each $R>2D>0$, if the above (4.3.1) and (4.3.2) hold on $X$, then there is an $\RCD^*(0, N)$-space $Y$ such that $X$ is $\Psi(\epsilon | N, D)$-measured Gromov-Hausdorff close to $\Bbb R\times_{e^r}Y$.
\end{Prop}

To prove Proposition 4.4, we need the following properties of Cheeger energy  (\cite[Theorem 6.8]{GMS}):
\begin{Prop}[\cite{GMS}]
Assume $(X_i, d_i, p_i)\in \op{CD}^*(K, N)$ that is measured Gromov-Hausdorff convergent to $(X, d, m)$. The following holds:

{\rm (4.5.1)} If $f_i\in L^2(X_i, m_i)$ is weakly convergent to $f\in L^2(X, m)$ and $\sup \int |f_i|^2dm_i < \infty$, then  
$$\liminf_{i\to \infty}Ch(f_i)\geq Ch(f);$$

{\rm (4.5.2)} For each $f\in L^2(X, m)$, there is a sequence $f_i\in L^2(X_i, m_i)$ such that $f_i$ is convergent to $f$ and 
$$Ch(f)=\lim_{i\to \infty} Ch(f_i).$$
\end{Prop}

\begin{proof}[Proof of Proposition 4.4]
By the compactness of $\RCD^*(K, N)$-spaces, consider a sequence of $\RCD^*(-(N-1), N)$-spaces, $(X_i, d_i, m_i, p_i)$ which is measured Gromov-Hausdorff convergent to $(X, d, m, p)$ and satisfies that for $\epsilon_i\to 0$, $R>2D$, there are $r_n\to \infty$,  $x_{n, R}^i\in B_D(p_i)$ such that 

(4.4.1) there is $p_n^i \in X_i$ such that $r_n-2R<d_i(x_{n, R}^i, p_n^i)<r_n+2R$ and for $f_{n, R}^i(y)=d_i(y, p_n^i)-d_i(x_{n, R}^i, p_n^i)$
$$-\kern-1em\int_{B_R(x_{n, R}^i)}\Delta f_{n, R}^i\geq N-1- \Psi(\epsilon_i, r_n^{-1} | N, D);$$

(4.4.2) for each $y\in B_R(x_{n, R}^i)$ there exists $q_n^i(y)\in X_i$ such that $d_i(q_n^i(y),  p_n^i)=r_n+5R$ and 
$$d(y, q_n^i(y))+d(y, p_n^i)\leq r_n+5R+\Psi(\epsilon_i, r_n^{-1} | N, D).$$

We will show that $X=\Bbb R\times_{e^r} Y$ where $Y$ is a $\RCD^*(0, N)$-space.

Assume $x_{n, R}^i\to x_R\in X$ as $n\to \infty, i\to \infty$. Since $f_{n, R}^i$ is $1$-Lipschitz, by Arzela-Ascoli theorem, there is $1$-Lipschitz function $f_R: B_{2R}(x)\to \Bbb R$, such that $f_{n, R}^i\to f_R$ as $n\to \infty, i\to \infty$. 

Claim 1: For $f_R$, we have that in $B_{R}(x_R)$,
$$\Delta f_R=N-1.$$

Let $\phi: X\to \Bbb R$ be a Lipschitz function with $\op{supp}(\phi) \subset\subset B_{2R}(x_R)$. By (4.5.2), there is a sequence of functions $\phi_i: X_i\to \Bbb R$ with $\op{supp}(\phi_i)\subset \subset B_{2R}(x_{n, R}^i)$, $\phi_i\to \phi$ and $Ch(\phi_i)\to Ch(\phi)$. By (4.5.1) and the increasing of 
$$\epsilon \mapsto \frac{|\nabla(g+\epsilon f)|_w^2-|\nabla g|_w^2}{2\epsilon}\overset{\epsilon\to 0}\to \Gamma(f, g)$$
 (cf. \cite{Gi2}), for each $\epsilon > 0$,
\begin{eqnarray*}
\int_X\Gamma(f_R, \phi)& \leq & \int_X \frac{|\nabla(\phi+\epsilon f_R)|_w^2-|\nabla \phi|_w^2}{2\epsilon}=\frac{1}{\epsilon}\left(Ch(\phi+\epsilon f_R)-Ch(\phi)\right)\\
&\overset{(4.5.1)}\leq &\liminf_{n, i\to \infty}\frac{1}{\epsilon}\left(Ch(\phi_i+\epsilon f_{n,R}^i)-Ch(\phi_i)\right)= \liminf_{n, i\to \infty}\int_{X_i} \frac{|\nabla(\phi_i+\epsilon f_{n,R}^i)|_w^2-|\nabla \phi_i|_w^2}{2\epsilon}.
\end{eqnarray*}
In above inequality, let $\epsilon\to 0$, then we have that
$$\int_X\Gamma(f_R, \phi)\leq \liminf_{n, i\to \infty}\int_{X_i}\Gamma(f_{n, R}^i, \phi_i).$$
Replacing $f_R$ and $f_{n, R}^i$ by $-f_R$ and $-f_{n, R}^i$ respectively, then by the same argument as above and the linearity of $\Gamma(\cdot, \cdot)$, we have 
$$-\int_X\Gamma(f_R, \phi)\leq -\limsup_{n, i\to \infty}\int_{X_i}\Gamma(f_{n, R}^i, \phi_i).$$
Then
$$\int_X\Gamma(f_R, \phi)= \lim_{n, i\to \infty}\int_{X_i}\Gamma(f_{n, R}^i, \phi_i).$$
By the generalized Laplacian comparison Theorem 2.6, for $y\in B_{2R}(x_{n, R}^i)$
$$\Delta f_{n, R}^i(y)\leq (N-1)\frac{\cosh d_i(y, p_n^i)}{\sinh d_i(y, p_n^i)}=N-1+\Psi(r_n^{-1} | N, D),$$
and if in addition $\phi,  \phi_i\geq 0$, we have that
$$-\int_X\Gamma(f_R, \phi)=\lim_{n, i\to \infty}\int_{X_i} \phi_id\Delta f_{n,R}^i\leq \lim_{n, i\to \infty}\int_{X_i} \phi_i (N-1+\Psi(r_n^{-1} | N, D))dm_i=\int_X \phi (N-1)dm.$$
Now by Proposition 2.2,  $\Delta f_R$ exists and 
$$\Delta f_R\leq N-1, \text{on } B_{2R}(x_R).$$

For the opposite inequality, take $\delta_j\to 0$ and a sequence of cut off functions as in \cite[Proposition 3.6]{CDNPSW}, $\psi_j: X \to \Bbb R$,
$$\psi_j(x)=\left\{\begin{array}{cc} \delta_j, & x\in \bar{B}_{R-\delta_j}(x_R);\\
R-d(x, x_R), & x\in A_{R-\delta_j, R}(x_R)=B_{R}(x_R)\setminus \bar{B}_{R-\delta_j}(x_R);\\
0, & \text{otherwise},\end{array}\right.$$
then $\psi_j\in W^{1,2}(X, d, m)$.
As above argument, take $\psi_{j}^i :X_i\to \Bbb R$, such that $\psi_j^i\to \psi_j$, $Ch(\psi_j^i)\to Ch(\psi_j)$ as $i\to \infty$ and $\left.\psi_j^i\right|_{\bar{B}_{R-\delta_j}(x_{n, R}^i)}=\delta_j$, $\psi_j^i(x)=0$ for $x\in X_i\setminus B_{R}(x_{n, R}^i)$ and $0\leq \frac{\psi_j^i}{\delta_j}\leq 1$. Then 
\begin{eqnarray*}
\int_{B_{R}(x_R)}\psi_j\Delta f_R & = & \int_{B_{R-\delta_j}(x_R)}\delta_j\Delta f_R + \int_{A_{R-\delta_j, R}(x_R)}\psi_j \Delta f_R\\
& = & \int_{A_{R-\delta_j, R}(x_R)}\Gamma(\psi_j, f_R)=\lim_{n, i\to \infty}\int_{A_{R-\delta_j, R}(x_{n, R}^i)}\Gamma(\psi_j^i, f_{n, R}^i)\\
& =& \lim_{n, i\to \infty}\int_{B_{R-\delta_j}(x_{n, R}^i)}\delta_j\Delta f_{n, R}^i+\int_{A_{R-\delta_j, R}(x_{n, R}^i)}\psi_j^i\Delta f_{n, R}^i.
\end{eqnarray*}
Divide $\delta_j$ in above inequality, 
$$\int_{B_{R-\delta_j}(x_R)}\Delta f_R + \int_{A_{R-\delta_j, R}(x_R)}\frac{\psi_j}{\delta_j} \Delta f_R=\lim_{n, i\to \infty}\int_{B_{R-\delta_j}(x_{n, R}^i)}\Delta f_{n, R}^i+\int_{A_{R-\delta_j, R}(x_{n, R}^i)}\frac{\psi_j^i}{\delta_j}\Delta f_{n, R}^i.$$
Since $0\leq \frac{\psi_j}{\delta_j}\leq 1$, we have that $\frac{\psi_j}{\delta_j}\Delta f_R$ is uniformly absolute continuous with respect to $\Delta f_R$, which is a radon measure. Thus 
$$\int_{A_{R-\delta_j, R}(x_R)}\frac{\psi_j}{\delta_j} \Delta f_R\to 0, \text{ as } \delta_j\to 0.$$
And similarly, 
$$\int_{A_{R-\delta_j, R}(x_{n,R}^i)}\frac{\psi_j^i}{\delta_j} \Delta f_{n,R}^i\to 0, \text{ as } \delta_j\to 0.$$

Now by above discussion and (4.4.1) and let $\delta_j\to 0$, we have that
$$\int_{B_R(x_R)}\Delta f_R=\lim_{n, i\to \infty}\int_{B_R(x_{n, R}^i)}\Delta f_{n, R}^i\geq (N-1)m(B_R(x_R)).$$
Together with $\Delta f_R\leq N-1$ on $B_R(x_R)$, we have that 
$$\Delta f_R=N-1, \text{ in } B_R(x).$$

Claim 2:  For each $a\in (-R, R)$, $af_R$ is $c$-concave and satisfies
$$(af_R)^c= -af_R-\frac{a^2}{2}.$$

As in \cite{Gi}, since $f_R$ is $1$-Lipschitz, for each $x, y\in X$,
$$af_R(x)-af_R(y)\leq |a|d(x, y)\leq \frac{a^2}{2}+\frac{d^2(x, y)}{2}.$$
Now by the definition of $c$-transform (see section 2.2),
$$(af_R)^c(y)=\inf_{x\in X}\frac{d^2(x,y)}{2}-af_R(x)\geq -af_R(y)-\frac{a^2}{2}.$$

For the opposite inequality, note that for each $y_n^i\in B_R(x_{n,R}^i)$, by (4.4.2), there is $q_n^i(y_n^i)\in X_i$ satisfying
$$d(q_n^i(y_n^i), p_n^i)= r_n+5R, \quad d(y_n^i, q_n^i(y_n^i))+d(y_n^i, p_n^i)\leq r_n+5R+\Psi.$$
Then 
\begin{eqnarray}
d(q_n^i(y_n^i), x_{n,R}^i)&\leq & d(q_n^i(y_n^i), y_n^i)+ d(y_n^i, x_{n,R}^i)\leq r_n+5R+\Psi- d(p_n^i, y_n^i)+ R \nonumber\\
&\leq & r_n+5R+\Psi -(d(p_n^i, x_{n,R}^i)-d(x_{n,R}^i, y_n^i))+R \leq r_n+5R+\Psi-r_n+2R + 2R \nonumber\\
&\leq & 9R+\Psi,
\end{eqnarray}
and
\begin{equation}d(q_n^i(y_n^i), y_n^i)\geq d(p_n^i, q_n^i(y_n^i))-d(p_n^i, x_{n, R}^i)-d(x_{n, R}^i, y_n^i)\geq r_n+5R-r_n-2R-R=2R,\end{equation}
where $\Psi=\Psi(\epsilon_i, r_n^{-1} | N, D)$. 

Fixed $y\in B_R(x_R)$, 
by (4.3), we may assume $y_n^i\to y$ and $q_n^i(y_n^i)\to q(y)$. Let $\gamma_n^i : [0, d(y_n^i, q_n^i(y_n^i))]\to X_i$ be a unit speed minimizing geodesic from $y_n^i$ to $q_n^i(y_n^i)$. 

For any $0\leq a\leq R$, by (4.4), we can take $\gamma_n^i(a)$ in $\gamma_n^i$. Now using (4.4.2) again,
$$r_n+5R\leq d(p_n^i, \gamma_n^i(a))+d(\gamma_n^i(a), q_n^i(y_n^i))\leq d(p_n^i, y_n^i)+a+d(\gamma_n^i(a), q_n^i(y_n^i))\leq r_n+5R+\Psi,$$
we have that
\begin{equation}d(p_n^i, y_n^i)+a\geq d(p_n^i, \gamma_n^i(a))\geq d(p_n^i, y_n^i)+a-\Psi.\end{equation}
Assume $\gamma_n^i(z)\to y_a\in X$, then
$$d(y_a, y)=a.$$
And by (4.5),
$$f_R(y_a)-f_R(y)=\lim_{n, i\to \infty} f_{n, R}^i(\gamma_n^i(a))-f_{n, R}^i(y_n^i)=\lim_{n, i\to \infty} d(p_n^i, \gamma_n^i(a))-d(p_n^i, y_n^i)=a.$$
By the definition of $c$-transform, we have that
$$(af_R)^c\leq \frac{d^2(y_a, y)}{2}-af_R(y_a)=\frac{a^2}{2}-a(a+f_R(y))=-\frac{a^2}{2}-af_R(y).$$

For the case $-R<a<0$, as the discussion in \cite{Gi, Gi2}, take a point $y_n^i(a)=\gamma'(d(p_n^i, y_n^i)+a)$ in the unit speed minimal geodesic $\gamma'$ from $p_n^i$ to $y_n^i$ as the behavior of $q_n^i(y_n^i)$. As the same argument as above, we have 
$$d(y_n^i(a), x_{n, R}^i)\leq d(y_n^i(a), y_n^i)+d(y_n^i, x_{n, R}^i)\leq 2R.$$
Assume $y_n^i(a)\to y'_a\in X$, 
then 
$$d(y'_a, y)=-a.$$
And thus
$$f_R(y'_a)-f_R(y)=a;$$
$$(af_R)^c\leq \frac{d^2(y'_a, y)}{2}-af_R(y'_a)=\frac{(-a)^2}{2}-a(a+f_R(y))=-\frac{a^2}{2}-af_R(y).$$

Claim 3:  $|D f_R|=1$ on $B_R(x)$, where $D f_R(y)=\limsup_{z\to y}\frac{|f_R(z)-f_R(y)|}{d(z, y)}$.

This argument is similar as the one in \cite{MN}. First, $|D f_R|\leq 1$. And by above argument, for each $y\in B_R(x)$, for $a\to 0$, there is $y_a$ such that $d(y_a, y)=a$ and $f_R(y_a)-f_R(y)=a$, then 
$$|D f_R|(y)=\limsup_{z\to y}\frac{|f_R(z)-f_R(y)|}{d(z, y)}\geq \lim_{a\to 0}\frac{|f_R(y_a)-f_R(y)|}{d(y_a, y)}=1.$$

Since for each $R$, there is $f_R$ on $B_R(x)$ such that the above three claims holds, by Arzela-Ascoli theorem, we may assume that $f_R\to f$. And $f$ satisfies that

(i) $\Delta f=N-1$ on $X$;

(ii) $|Df|=1$ on $X$;  And by \cite{Cheeger}, for $\RCD^*$-space $|Df|=|\nabla f|_w$.

And for each $a \in (-R, R)$, since $f$ is $1$-Lipschitz and $f_R\to f$,
$$(af)^c=\inf_{x\in X}\left(\frac{d^2(x, y)}{2}-af(x)\right)\leq \inf_{x\in B_R(x_R)}\left(\frac{d^2(x, y)}{2}-af(x)\right)\leq -\frac{a^2}{2}-af(y),$$
we have that

(iii) for each $a\in \Bbb R$, $af$ is $c$-concave and $(af)^c=-af-\frac{a^2}{2}$.

Now by Theorem 4.2, $X$ is isometric to a warped product space $\Bbb R\times_{e^r} Y$, $Y$ is an $\RCD^*(0, N)$-space. \end{proof}

Note that in the proof of \cite[Theorem 1.2]{CDNPSW}, they used \cite{AT}'s result to derive the existence of regular Lagrangian flow and then use the regular Lagrangian flow to define the map $\Bbb R\times \tilde X\to \tilde X$. Here one may use the $c$-concavity of $a f$ for each $a\in \Bbb R$ and the results in \cite{GRS} to construct the map $\Bbb R\times \tilde X\to \tilde X$ (cf. \cite{Gi2}).

To prove Proposition 4.3, note that it is a kind of quantitative version of \cite[Lemma 3.4]{CDNPSW} and thus their proofs are similar. Here we will give a rough proof of Proposition 4.3 following the same line as \cite[Lemma 3.4]{CDNPSW}. 

First as \cite[Proposition 3.5]{CDNPSW}, we have 
\begin{Lem}
Let $(X, d, m)$ be as in Proposition 4.3 and let $(\tilde X, \tilde d, \tilde m)$ be its universal cover. Then for $\tilde p\in \tilde X$, $R>2D$, there exist $r_n\to \infty$ such that
$$\lim_{n\to \infty}\frac{s_{\tilde m}(\tilde p, r_n+5R)}{s_{\tilde m}(\tilde p, r_n-5R)}\geq e^{10R(N-1-\epsilon)}.$$
\end{Lem}
The proof of Lemma 4.6 is the same as \cite[Lemma 4.2]{CRX} where they only used the almost maximal volume entropy condition.

And by \cite[Proposition 3.6]{CDNPSW}, for $r(y)=d(\tilde p, y)$ and for all but countably $t_2\geq t_1>0$,
\begin{equation}\int_{A_{t_1, t_2}(\tilde p)}\Delta r=s_{\tilde m}(\tilde p, t_2)-s_{\tilde m}(\tilde p, t_2).\label{lap-vol}\end{equation}
 By Lemma 4.6, Theorem 2.6 and \eqref{lap-vol}, we have that
\begin{Lem}
Let $\tilde X$ be as in Lemma 4.6 and let $r(y)$ be as above. Then for $A_n=A_{r_n-5R, r_n+5R}(\tilde p)$,
$$\left|-\kern-1em\int_{A_n}\Delta r-(N-1)\right|\leq \Psi(\epsilon, r_n^{-1} |N, R).$$
\end{Lem}

Now fixed $n$. Let $E$ be the maximal subset of $\tilde x_i \in A_n$ such that for $i\neq j$, $B_{R}(\tilde x_i)\cap B_R(\tilde x_j)=\emptyset$. We will find a subset $E'\subset E$, such that for each $\tilde x_i\in E'$, (4.3.1) and (4.3.2) hold.

Let $F=\bigcup_{\tilde x_i\in E}B_R(\tilde x_i)$, then by Theorem 2.4 and the doubling property, we have that 
\begin{equation}\frac{\tilde m(F)}{\tilde m(A_n)}\geq c(N, R)>0. \end{equation}

Let $S=\{x\in A_n,\, \exists z\in \tilde X, d(\tilde p, z)=d(\tilde p, x)+d(x, z)=r_n+5R\}$. Let 
$E_1(\eta)=\{\tilde x_i\in E,\, \frac{\tilde m(B_R(\tilde x_i)\setminus S)}{\tilde m(B_R(\tilde x_i))}<\eta\}$ and let $F_1(\eta)=\bigcup_{\tilde x_i\in E_1(\eta)}B_R(\tilde x_i)$. Then again by Theorem 2.4 and Lemma 4.6, as \cite[Lemma 5.5]{CRX}, we have that
\begin{equation}\frac{\tilde m(F_1(\eta))}{\tilde m(F)}\geq 1-\eta^{-1}\Psi(\epsilon | N, R). \end{equation}

Let $E_2(\eta)=\left\{\tilde x_i\in E,\, \left|-\kern-1em\int_{B_R(\tilde x_i)}\Delta r-(N-1)\right|\leq \eta^{-1}\Psi(\epsilon, r_n^{-1} |N, R)\right\}$ and let $F_2(\eta)=\bigcup_{\tilde x_i\in E_2(\eta)}B_R(\tilde x_i)$. Then by Lemma 4.7 and (4.7), as \cite[Lemma 5.8]{CRX},
\begin{equation}\frac{\tilde m(F_2(\eta))}{\tilde m(F)}\geq 1-\eta. \end{equation}

Take $\eta$ suitable small and let $E'=E_1(\eta)\cap E_2(\eta)$. Now by  (4.7)-(4.9) and consider a deck transformation of $\tilde X$, Proposition 4.3 is derived.

\subsection{Non-collapsing and Free limit group action}

Let $(X_i, d_i, m_i)$ be as in (4.1), (4.2). Then by Theorem 3.1 and Theorem 4.1, we have the following diagram:
\begin{equation}\begin{array}[c]{ccc}
(\tilde X_i,\tilde p_i,\Gamma_i)&\xrightarrow{GH}&(\Bbb H^k,\tilde p,G)\\
\downarrow\scriptstyle{\tilde \pi_i}&&\downarrow\scriptstyle{\tilde \pi}\\
(\hat X_i, \hat p_i, \hat \Gamma_i) & \xrightarrow{GH} & (\hat X, \hat p, \hat G)\\
\downarrow\scriptstyle{\hat \pi_i}&&\downarrow\scriptstyle{\hat \pi}\\
(X_i,p_i)&\xrightarrow{GH} &(X, p)
\end{array}\end{equation}
where by Theorem 3.1, there exist $\Gamma_{i\epsilon}\to G_0$, $G_0$ is the identity component, $\hat X_i=\tilde X_i/\Gamma_{i\epsilon}$, $\hat \Gamma_i=\Gamma_i/\Gamma_{i\epsilon}$, $\hat p_i=\tilde \pi_i(\tilde p_i)$, $\hat X=\Bbb H^k/G_0$, $\hat G=G/G_0$ and $\hat \Gamma_i\overset{isom}\cong \hat G$.

In this subsection, we will show that under the additional condition (a) or (b), the isometric action $G$ is discrete and free. Recall that

(a) $\{X_i\}$ are Alexandrov spaces with curvature $\geq -1$ (which implies Theorem 1.1);

(b) $\op{Haus}^N(X_i)>v>0$ for all $i$ (which implies Theorem 1.2).

First, we show that $G$ is discrete.

 Condition(b) implies that $G_0=\{e\}$ and thus $G$ is discrete.  With the condition (a), we will use the following generalized Margulis lemma in Alexandrov spaces proved by Xu-Yao \cite{XY}.
\begin{Thm}[\cite{XY}]Given $N$, there are $\epsilon(N), c(N)>0$, such that for an $N$-dimensional Alexandrov space $X$ with curvature $\geq -1$, for $x\in X$, the subgroup of fundamental group at $x$ that generated by loops of length $\leq \epsilon(N)$ contains a nilpotent subgroup with index $\leq c(N)$.
\end{Thm}
By Theorem 4.8, and (4.10), we know that $G_0$ is nilpotent. And \cite[Theorem 2.5]{CRX} says that there is no nontrivial connected nilpotent isometric group on $\Bbb H^k$, i.e., $G$ is discrete. Note that if one can show the generalized Margulis lemma in $\RCD^*$-spaces, one can drop the condition (b) in Theorem 1.2.

Second, using the same argument as \cite[Theorem 4.7]{CRX}, we know that $k=N$.

Now, to complete Theorem 1.1 and 1.2, we only need to show that the discrete group $G$ acts freely on $\Bbb H^N$.

For a metric measure space $(X, d, m, x)$ and $r>0$, consider the rescaled and renormalized metric measure space  $(X, r^{-1}d, m_r^x, x)$, where 
$$m_r^x=\left(\int_{B_r(x)}1-\frac{d(x, y)}{r}dm(y)\right)^{-1}m.$$
If $(X, d, m)\in\RCD^*(K, N)$, then $(X, r^{-1}d, m_r^{x})\in \RCD^*(r^2K, N)$.

Recall that in the proof of the free limit action \cite[Theorem 2.1]{CRX}, we need two fact: (i) the universal cover is non-collapsing and the volume convergence theorem holds which guarantees that the limit action is trivial iff the convergent sequence of isometric actions is trivial; (ii) the uniform Reifenberg condition (i.e., there is $(\epsilon_i, l_i)$, such that for each $\tilde y_i\in \tilde X_i$, $d_{GH}(B_s(\tilde y_i), \underline{B}_s^0)\leq \epsilon_i s$ for $s<l_i$ where $\underline{B}_s^0\subset \Bbb R^N$) which guarantees that for each point $\tilde y_i\in \tilde X_i$ with any rescalings  $r_i\to 0$, $(X_i, r_i^{-1}d_i, y_i)\overset{GH}\to (\Bbb R^N, d_{\Bbb R^N}, 0)$. And then we could use the properties of isometric actions on $\Bbb R^N$ to derive that the limit isometric group action has no fixed point. In (4.10), the fact (i) holds for we have known $k=N$. For the fact (ii), we have the following,

\begin{Thm}
Assume a sequence of $\RCD^*(K, N)$-spaces, $(X_i, d_i, m_i, x_i)$ is measured Gromov-Hausdroff convergent to an $N$-dimensional Riemannian manifold $(M, g, x)$ with sectional curvature bounded above by $H'$. Then for each sequence $r_i\to 0$, if $(X_i, r^{-1}_i d_i, m_{r_i}^{x_i}, x_i)$ is measured Gromov-Hausdorff convergent to $(Y, d, m, y)$, we have that $(Y, d, m, y)=(\Bbb R^n, d_{\Bbb R^N}, c\op{Haus}^N, 0)$, where $c=\left(\int_{B_1(0)}1-|y|d\op{Haus}^N(y)\right)^{-1}$.
\end{Thm}
\begin{proof}
Since $(X_i, d_i, m_i, x_i)\overset{mGH}\to (M, g, x)$, by the compactness of $\RCD^*(K, N)$ spaces and \cite[Theorem 1.7]{St2} we have that $(M, g)$ has Ricci curvature $\geq K$ and for $0<r<R<1$,
$$\left|\frac{m_i(B_r(x_i))}{m_i(B_R(x_i))}-\frac{\volume(B_r(x))}{\volume(B_R(x))}\right|\to 0,\, i\to \infty.$$

Then as $i\to \infty$, by volume comparison, we have that 

\begin{eqnarray*}
\left|\frac{m_{r_i}^{x_i}(B_r(x_i))}{m_{r_i}^{x_i}(B_R(x_i))}-\left(\frac{r_ir}{r_iR}\right)^N\right| & \leq &\left|\frac{m_{r_i}^{x_i}(B_r(x_i))}{m_{r_i}^{x_i}(B_R(x_i))}-\frac{\volume(B_{r_ir}(x))}{\volume(B_{r_iR}(x))}\right|\\
& & +\max\left\{\left|\frac{\svolball{H}{r_ir}}{\svolball{H}{r_iR}}-\left(\frac{r}{R}\right)^N\right|, \left|\frac{\svolball{H'}{r_ir}}{\svolball{H}{r_iR}}-\left(\frac{r}{R}\right)^N\right|\right\}\\
&\to & 0,
\end{eqnarray*}
where $\svolball{H}{r}$ is the volume of the ball with radius $=r$ in the simply connected $N$-space form with sectional curvature $=H$ and $H=\frac{K}{N-1}$.
And thus 
$$\frac{m(B_r(y))}{m(B_R(y))}=\left(\frac{r}{R}\right)^N.$$

Note that $(Y, d, m)\in \RCD^*(0, N)$ and the above equality holds for $R\to \infty$ ($Rr_i<1$). Thus by Theorem 2.8, we have that $(Y, d, m, y)$ is isometric to $(C(Z), d_{C(Z)}, m_{C(Z)}, y)$, where $Z\in \RCD^*(N-2, N-1)$, $y$ is the cone vertex. 

If $C(Z)$ is not isometric to $\Bbb R^N$, then as the discussion in \cite[Theorem 9.69]{Ch}, take $z\in C(Z)$, $z_i\in (X_i, r_i^{-1}d_i, m_{r_i}^{x_i})$, such that $z_i\to z$. Let $s_i=r_i^{-1}d_i(z_i, x_i)$. Then $s_i\to d_{C(Z)}(z, y)=s$.  As above we have that 
$$\frac{m_{r_i}^{x_i}(B_{r+s_i}(z_i))}{m_{r_i}^{x_i}(B_{R+s_i}(z_i))}\to \left(\frac{r+s}{R+s}\right)^N.$$
Again by Theorem 2.8, passing to a subsequence, we have $(X_i, r_i^{-1}d_i)\supset B_{R+s_i}(z_i)\to B_{R+s}(z)$, where $B_{R+s}(z)$ is a ball in a cone with vertex $z$. Thus $C(Z)$ can be also viewed as a cone with vertex at $z$. Then a ray from $z$ to $y$ can be extended to a line. By splitting Theorem 2.7, we have $C(Z)=\Bbb R^N$, a contradiction.
\end{proof}

\end{document}